\documentclass{article}
\usepackage{amsmath, amsfonts, amsthm, amssymb}
\usepackage{graphicx}
\usepackage{float}
\usepackage{verbatim}% for comment
\usepackage[usenames]{color}

\numberwithin{equation}{section}
\newtheorem{thm}{Theorem}[section]
\newtheorem{lma}[thm]{Lemma}
\newtheorem{cor}[thm]{Corollary}

\newtheorem{prop}[thm]{Proposition}

\theoremstyle{definition}
\newtheorem{defn}[thm]{Definition}

\renewcommand{\geq}{\geqslant}
\renewcommand{\leq}{\leqslant}
\renewcommand{\H}{\text{H}}
\renewcommand{\P}{\text{P}}

\date{}

\title{Dimension and measure for generic continuous images}

\author{R. Balka$^{1}$, \'A. Farkas$^{2}$, J. M. Fraser$^{2}$ and J. T. Hyde$^{2}$}

\begin{document}
\maketitle

\begin{center}
$^1$Alfr\'ed R\'enyi Institute of Mathematics, Hungarian Academy of Sciences, PO Box 127, H-1364 Budapest, Hungary\\ \vspace{3mm}
$^2$Mathematical Institute, University of St Andrews, North Haugh, St Andrews, Fife, KY16 9SS, Scotland
\end{center}
\vspace{0mm}
\begin{abstract}
We consider the Banach space consisting of continuous functions from an arbitrary uncountable compact metric space, $X$, into $\mathbb{R}^n$.  The key question is `\emph{what is the generic dimension of} $f(X)$?' and we consider two different approaches to answering it:  Baire category and prevalence.  In the Baire category setting we prove that typically the packing and upper box dimensions are as large as possible, $n$, but find that the behaviour of the Hausdorff, lower box and topological dimensions is considerably more subtle.  In fact, they are typically equal to the minimum of $n$ and the \emph{topological} dimension of $X$.  We also study the typical Hausdorff and packing measures of $f(X)$ and, in particular, give necessary and sufficient conditions for them to be zero, positive and finite, or infinite.

It is interesting to compare the Baire category results with results in the prevalence setting.  As such we also discuss a result of Dougherty on the prevalent topological dimension of $f(X)$ and give some simple applications concerning the prevalent dimensions of graphs of real-valued continuous functions on compact metric spaces, allowing us to extend a recent result of Bayart and Heurteaux.
\end{abstract}

\section{Introduction}

Let $X$ be an uncountable compact metric space, $n$ be a positive integer and let $C_n(X)$ denote the set of continuous functions from $X$ to $\mathbb{R}^n$, which is a Banach space over $\mathbb{R}$ when equipped with the supremum norm, $\| \cdot \|_\infty$.  We investigate the dimensions of the image of $X$ under mappings from $C_n(X)$.  Rather than compute these dimensions for specific examples, we look to find the `generic answer', with our key question being:
\[
\emph{What is the dimension of $f(X)$ for a generic $f \in C_n(X)$?}
\]
In order to do this we need a suitable notion of genericity in Banach spaces, which we obtain using the theories of \emph{Baire category} and \emph{prevalence}, see Subsection \ref{prelims} for an account of the theory.
\\ \\
Our main results concern the Baire category setting.  We prove that typically the packing and upper box dimensions are as large as possible, $n$, but find that the behaviour of the Hausdorff, lower box and topological dimensions is considerably more subtle.  In fact they are typically equal to the minimum of $n$ and the \emph{topological} dimension of $X$.  Interestingly, this means that the typical Hausdorff dimension is usually not as small as possible, although it is always an integer.  This is in stark contrast to many other results concerning Baire category and Hausdorff dimension, where one normally sees that the typical Hausdorff dimension is as small as possible, see for example \cite{graphsums, me_typrandom}.  During our investigation, we are able to generalise some results of Kato \cite{kato} concerning the typical topological dimension of $f(X)$.
\\ \\
We also investigate the typical Hausdorff and packing measures of $f(X)$ in the appropriate dimensions.  We find an interesting dichotomy. If the topological dimension of $X$ is greater than or equal to $n$, then the typical $n$-dimensional packing and Hausdorff measures are positive and finite.  However, if the topological dimension of $X$ is some number $t$ strictly less than $n$, then the typical Hausdorff dimension of $f(X)$ is $t$ but the typical $t$-dimensional Hausdorff measure is infinity and the typical packing dimension of $f(X)$ is $n$ but the typical $n$-dimensional packing measure is zero.  A similar dichotomy was observed by Fraser when studying typical random self-similar sets \cite[Theorem 2.5]{me_typrandom}.
\\ \\
It is natural and interesting to examine the same questions in the setting of prevalence.  In the following subsection we observe that these are answered by a result of Dougherty \cite{doug} and in particular for a \emph{prevalent} set of functions in $C_n(X)$, the topological, Hausdorff, packing and box dimensions of $f(X)$ are as large as possible, namely $n$, and we discuss some simple applications of this fact.  In particular, we obtain results on the prevalent dimensions of graphs of real-valued continuous functions on compact metric spaces which allow us to extend a recent result of Bayart and Heurteaux \cite{BH}.  Their result is stated below as Theorem \ref{BH1} and we provide a strengthening of this, Theorem \ref{fullprev}.

\subsection{Genericity and dimension} \label{prelims}

In this subsection we introduce some preliminary concepts and notation which will be required to state our results.  In particular, we discuss Baire category, prevalence and the different notions of dimension we will be concerned with.  To put our results in context, we also discuss some previous work related to our results and, in particular, we apply a result of Dougherty in the prevalence setting.  The main results of this paper concern the Baire category setting and will be stated in Section \ref{results}.
\\ \\
Baire category provides an important way of describing the \emph{generic} behavior of elements in a Banach space.  We will recall the basic definitions and theorems.  For more details, see \cite{oxtoby}.
\begin{defn}
Let $(X,d)$ be a complete metric space (for our purposes $X$ will be a Banach space).  A set $M\subseteq X$ is said to be \emph{of the first category}, or, \emph{meagre}, if it can be written as a countable union of nowhere dense sets and a set $T \subseteq X$ is \emph{residual}, or, \emph{co-meagre}, if $X \setminus T$ is meagre.  Finally, a property is called \emph{typical} if the set of points which have the property is residual.
\end{defn}
In Subsections \ref{keytypproofs}--\ref{typproofs3} we will use the following theorem to test for typicality without mentioning it explicitly.
\begin{thm}[Baire Category Theorem] \label{BCThm}
In a complete metric space, $X$, a set $T \subseteq X$ is residual if and only if $T$ contains a countable intersection of open dense sets or, equivalently, $T$ contains a dense $\mathcal{G}_\delta$ subset of $X$.
\end{thm}
For a proof of Theorem \ref{BCThm}, see \cite[Theorem 9.2]{oxtoby}.  Prevalence provides another important way of describing the \emph{generic} behavior of elements in a Banach space.  In finite dimensional vector spaces Lebesgue measure provides a natural tool for deciding if a property is `generic'.  Namely, if the set of elements which do not have some property is a Lebesgue null set, then it is said that this property is `generic' from a measure theoretical point of view.  However, when the space in question is infinite dimensional this approach breaks down because there is no useful analogue to Lebesgue measure in the infinite dimensional setting.  The theory of prevalence has been developed to solve this problem.  It was formulated by Hunt, Sauer and Yorke in 1992 \cite{prevalence1} in the context of completely metrizable topological vector spaces, see also \cite{prevalence}. For the purposes of this paper we will only set up the theory for Banach spaces.  We note that Christensen introduced a similar theory in the 1970s \cite{christ, christ2} in the setting of abelian Polish groups.
\begin{defn} \label{prevalentdef}
Let $X$ be a Banach space.  A Borel set $S \subseteq X$ is called \emph{shy} if there exists a compactly supported Borel probability measure $\mu$ on $X$ such that $\mu\left(S+x\right) = 0$ for all $x \in X$.  A non-Borel set $S \subseteq X$ is \emph{shy} if it is contained in a shy Borel set and the complement of a shy set is called a \emph{prevalent} set.
\end{defn}
Both prevalence and typicality are reasonable notions of genericity and satisfy many of the natural properties one would expect from such a notion.  Perhaps most importantly they are both stable under taking countable intersections.  Interestingly, however, they often give starkly different answers to genericity questions and as such their interaction and differences have attracted a lot of attention in the literature in recent years.  When using these notions to search for the generic value of a limiting procedure (as in our situation), roughly speaking, one expects typicality to favour divergence and prevalence to favour convergence.  A fascinating example of this behaviour, and one which provides a poignant illustration of the differences in the two theories, is provided by normal numbers.  In particular, the set of normal numbers in the unit interval is prevalent (it is of full Lebesgue measure; a simple consequence of the Ergodic Theorem/Strong Law of Large Numbers), but it is also meagre, see \cite{prevalence, salat}.  So, \emph{prevalently} frequencies of digits in decimal expansions converge and \emph{typically} they diverge.
\\ \\
In this paper we will be concerned with five different notions of dimension, namely, the topological, Hausdorff, packing, and upper and lower box dimensions, which we will denote by $\dim_\text{T}$, $\dim_\text{H}$, $\dim_\text{P}$, $\overline{\dim}_\text{B}$ and $\underline{\dim}_\text{B}$ respectively, as well as the Hausdorff and packing measure, which we will denote by $\mathcal{H}^s$ and $\mathcal{P}^s$ respectively for $s \geq 0$.  Rather than define each of these individually, we refer the reader to \cite{falconer, hurwal, mattila} for definitions and basic properties.  The following proposition gives the relationships between these dimensions, which will be used throughout the paper without being mentioned explicitly.  For the reader's convenience, any other basic properties of these dimensions will be introduced when required during the various proofs.

\begin{prop} \label{relationships}
For a non-empty bounded subset $F$ of a metric space $X$ we have the following relationships between the dimensions discussed above:
\[
\begin{array}{ccccccc}
 & &                                                            &&                  \dim_\text{\emph{P}} F                   & &   \\
&          &&                       \rotatebox[origin=c]{45}{$\leq$}          & &              \rotatebox[origin=c]{315}{$\leq$} &   \\
 \dim_\text{\emph{T}} F & \leq & \dim_\text{\emph{H}} F                                            & &        &&                  \overline{\dim}_\text{\emph{B}} F  \\
 &                 &&                \rotatebox[origin=c]{315}{$\leq$}              &&           \rotatebox[origin=c]{45}{$\leq$} &  \\
 & &                           &&                                          \underline{\dim}_\text{\emph{B}} F                         & &
\end{array}
\]
and, moreover, unlike the other dimensions, $\dim_\text{\emph{T}} F$ is always a non-negative integer or $+\infty$.  We also have $\mathcal{H}^s(F) \leq \mathcal{P}^s(F)$ for all $s \geq 0$.
\end{prop}

Prevalence and Baire category have been used extensively in the literature to study dimensional properties of generic continuous functions.  In particular, there has been considerable interest in studying the generic dimension of images of continuous functions.  This problem is related to the seminal results of Kaufman \cite{kaufman}, Mattila \cite{mattilaproj} and Marstrand \cite{marstrand} on the almost sure dimension of orthogonal projections of sets in Euclidean spaces.  For example, Marstrand's projection theorem states that if $F \subseteq \mathbb{R}^2$ is Borel, then for almost all linear subspaces of the plane, the Hausdorff dimension of the corresponding orthogonal projection is equal to $\min\{1, \dim_\H F\}$, i.e., the dimension is generically preserved.  Recently, Orponen \cite{orpproj} has obtained interesting results on generic projections in the Baire category setting.  In \cite{prevalentimages}, prevalence was used to extend the results on projections to the space of all continuously differentiable maps.  Again it was found that the dimension is generically preserved, this time using prevalence to give a notion of `generic'.  It is natural to ask the same question in the much larger space of just continuous functions.  This question can be answered in the prevalence setting by the following result of Dougherty, see \cite[Theorem 11]{doug}.
\begin{thm}[Dougherty] \label{doug1}
If $K$ is homeomorphic to the Cantor space, then
\[
\{ f \in C_n(K) : \text{\emph{int}}\left(f(K) \right) \neq \emptyset \}
\]
is a prevalent subset of $C_n(K)$.
\end{thm}

This result has the following immediate corollary.

\begin{cor}\label{doug2}
The following properties are prevalent in the space $C_n(X)$:
\begin{itemize}
\item[(1)] $\text{\emph{int}}\left(f(X) \right) \neq \emptyset$;
\item[(2)] $\dim_\text{\emph{T}} f(X) = \dim_\text{\emph{H}} f(X) = \dim_\text{\emph{P}} f(X) = \underline{\dim}_\text{\emph{B}} f(X) = \overline{\dim}_\text{\emph{B}} f(X) = n$;
\item[(3)] $0 < \mathcal{H}^n(f(X)) = \mathcal{P}^n(f(X)) < \infty$.
\end{itemize}
\end{cor}

\begin{proof}
Part (1) follows from Theorem \ref{doug1} since all uncountable compact metric spaces contain a subset $K \subseteq X$  homeomorphic to the Cantor space, see \cite[Theorem 11.11]{realanalysis}.  The result then follows since \cite[Proposition 8]{doug} and Theorem \ref{doug1}   imply  that the set of $f \in C_n(X)$  such that $ \text{{int}}\left(f(K) \right)  \neq \emptyset$ is prevalent.  Parts (2) and (3) follow immediately from part (1) and the fact that $\mathcal{H}^n=\mathcal{P}^n$ for the Borel subsets of $\mathbb{R}^n$, see \cite[Theorem 6.12]{mattila}.
\end{proof}

The fact that the $n$-dimensional Hausdorff measure of the image is prevalently positive and finite answers a question posed by Pablo Shmerkin to one of the authors in April 2012.  Contrary to the continuously differentiable case, the above corollary shows that in the space $C_n(X)$ the Hausdorff dimension is not preserved and in fact it is `almost surely' as large as possible.  The generic topological dimension of $f(X)$ has been studied in the Baire category setting by Kato \cite[Proposition 3.6 and Theorem 4.6]{kato}.

\begin{thm}[Kato] \label{kato1}
If $\dim_\text{\emph{T}} X < n$, then the set
\[
\{ f \in C_n(X) : \dim_\text{\emph{T}} f(X) \leq \dim_\text{\emph{T}} X \}
\]
is residual. If $\dim_\text{\emph{T}} X \geq n$, then the set
\[
\{ f \in C_n(X) : \dim_\text{\emph{T}} f(X) = n \}
\]
is residual.
\end{thm}

In fact, the second statement of the above theorem has older origins dating back to Hurewicz-Wallman and Alexandroff. If $f\in C_n(X)$, then $y\in f(X)$ is a \emph{stable value} of $f$ if there exists $\varepsilon>0$ such that for all $g\in B(f,\varepsilon)$ we have $y\in g(X)$. Clearly, if $y\in \mathbb{R}^n$ is a stable value of $f$ then $B(y,\varepsilon/2)\subseteq g(X)$ for all $g\in B(f,\varepsilon/2)$.  Hurewicz  and  Wallman \cite{hurwal} show that if $\dim_\text{T} X\geq n$, then there exists an $f\in C_n(X)$ that has a stable value. Thus it is enough to prove that such $f$s are dense, but this is straightforward.
\\ \\
Also, let $f\colon X\to B^n$ be an onto map, where $B^n$ denotes the closed unit ball in $\mathbb{R}^n$. Then $f$ is an \emph{essential map} if, whenever $g=f$ on $f^{-1}(\partial B^n)$, then $B^n \subseteq g(X)$.  Alexandroff \cite{A} shows that if $\dim_\text{T} X \geq n$ then there exists an essential map $f\in C_n(X)$, and one can use methods from algebraic topology to show that essential maps have stable values.
\\ \\
In this paper we examine the typical dimension and measure of $f(X)$.  We obtain precise results for all the notions of dimension described above and give necessary and sufficient conditions for the appropriate Hausdorff and packing measures to be positive and finite.  Interestingly, the `dimension preservation principle' holds in the typical case for the topological dimension, but not in general for any of the other dimensions.  In the topological dimension case, we obtain a sharpening of the above result of Kato, see Theorems \ref{typmain} and \ref{hausdecomp}.  In particular, the typical topological dimension of $f(X)$ is precisely $\min\{n,\dim_\text{T} X \}$.
\\ \\
Over the past 15 years there has also been considerable interest in studying the prevalent and typical dimensions of graphs of continuous real-valued functions, see \cite{ BH, me_horizon, me_prevalence, lowerprevalent, bairefunctions, humkepacking, graphsums, mcclure, shaw}, where the graph of $f \in C_n(X)$ is defined as
\[
G_f = \left\{ (x,f(x)) : x \in X \right\} \subseteq X \times \mathbb{R}^n.
\]
The most general result in the case of prevalence to date has been given by Bayart and Heurteaux \cite{BH}.

\begin{thm}[Bayart-Heurteaux] \label{BH1}
Let $X$ be a compact subset of $\mathbb{R}^m$ with positive Hausdorff dimension.  The set
\[
\{ f \in C_1(X) : \dim_\text{\emph{H}} G_f = \dim_\text{\emph{H}} X+1 \}
\]
is a prevalent subset of $C_1(X)$.
\end{thm}

The case where $X = [0,1]^m$ was proved by Fraser and Hyde \cite{me_prevalence}.  The method of proof used in \cite{BH} was to use fractional Brownian motion on $X$.  The assumption that $X$ has positive Hausdorff dimension was needed to guarantee the existence of an appropriate measure to use in the energy estimates.  Interestingly, this left open the case where $\dim_\H X = 0$.  Clearly if $X$ is finite or countable then the dimension of the graph is necessarily 0, so the only open case is when $X$ has cardinality continuum but is zero dimensional.  In this case, one can compute the prevalent dimension of the \emph{graph} by considering the prevalent dimension of the \emph{image} and so the study of graphs on zero dimensional sets falls naturally into our investigation.  We observe that the problem can be solved by applying Corollary \ref{doug2}.  In particular, one obtains:

\begin{cor} \label{0case}
Suppose $\dim_\text{\emph{H}} X = 0$.  Then the set
\[
\{f \in C_n(X) : \dim_\text{\emph{H}} G_f  = n \}
\]
is prevalent. If $\dim_\text{\emph{P}} X = 0$, then the set
\[
\{f \in C_n(X) :\dim_\text{\emph{H}} G_f  =   \dim_\text{\emph{P}} G_f = n \}
\]
is also prevalent.
\end{cor}

\begin{proof}
This follows immediately from Corollary \ref{doug2} and the projection and product formulae for Hausdorff and packing dimension, see \cite[Chapter 6]{falconer} and \cite{products}.  In particular, the image $f(X)$ is the projection of the graph $G_f$ onto $\mathbb{R}^n$ and we obtain that for all $f \in C_n(X)$ we have
\[
 \dim_\text{H} f(X) \leq \dim_\text{H} G_f \leq \dim_\text{H} (X \times \mathbb{R}^n ) \leq  \dim_\text{\H} X + n
\]
and
\[
\dim_\text{P} f(X) \leq \dim_\text{P} G_f \leq \dim_\text{P} (X \times \mathbb{R}^n) \leq \dim_\text{P} X + n
\]
and since for a prevalent $f \in C_n(X)$ we have $ \dim_\text{H} f(X) =  \dim_\text{P} f(X) = n$, the result follows.
\end{proof}

A combination of Corollary \ref{0case} and the result of Bayart and Heurteaux gives.

\begin{thm} \label{fullprev}
Let $X$ be an uncountable compact subset of $\mathbb{R}^m$.  The set
\[
\{ f \in C_1(X) : \dim_\text{\emph{H}} G_f  = \dim_\text{\emph{H}} X +1 \}
\]
is a prevalent subset of $C_1(X)$.  If $X$ is finite or countable, then $\dim_\text{\emph{H}} G_f = 0$ for all $f \in C_1(X)$.
\end{thm}

We remark here that a compact subset of $\mathbb{R}^m$ is either finite, countable or has cardinality continuum (see \cite{cies}, Corollary 6.2.5).

\section{Results in the Baire category setting} \label{results}

This is the main section of the paper where we will state our results on the typical dimension and measure of $f(X)$.  The proofs are deferred to the subsequent section.  Our first result concerns the typical dimensions of the image of $X$.

\begin{thm} \label{typmain}
For $f \in C_n(X)$, the properties
\[
\dim_\text{\emph{T}} f(X) =\dim_\text{\emph{H}} f(X) = \underline{\dim}_\text{\emph{B}}   f(X) = \min\{n,  \, \dim_\text{\emph{T}} X\}
\]
and
\[
\dim_\text{\emph{P}} f(X) = \overline{\dim}_\text{\emph{B}}  f(X)  = n
\]
are typical.
\end{thm}

We actually obtain finer information about the topological structure of $C_n(X)$ in terms of dimensions of images from which Theorem \ref{typmain} follows immediately, see Theorems \ref{hausdecomp} and \ref{packdecomp}.  We note an interesting corollary of Theorem \ref{typmain} is that the typical Hausdorff dimension is not in general as small as possible, but is always an integer.  At first sight this may be surprising as one often finds that \emph{typically} the Hausdorff dimension \emph{is} as small as possible, see for example  \cite{graphsums, me_typrandom}.  However, in our situation a more complex phenomenon is taking place.  The fact that the typical box dimensions are integers was suggested in \cite{mirzaie} although the proofs given there on the typical box dimensions are incorrect.
\\ \\
Our next two results give a precise topological description of $C_n(X)$ in terms of dimensions of images.
\begin{thm} \label{hausdecomp}
Let $\dim$ denote $\dim_\text{\emph{T}}, \, \dim_\text{\emph{H}}$ or $\underline{\dim}_\text{\emph{B}}$.  Then $C_n(X)$ is a disjoint union of the following three sets:
\begin{eqnarray*}
C_n(X) &=& \left\{ f \in C_n(X)  : 0 \leq \dim f(X) <  \min\{n,  \, \dim_\text{\emph{T}} X\} \right\}\\ \\
&\,& \quad \cup \ \left\{ f \in C_n(X)  :  \dim f(X) = \min\{n,  \, \dim_\text{\emph{T}} X \}\right\}\\ \\
&\,&  \quad \cup \ \left\{ f \in C_n(X)  :  \min\{n,  \, \dim_\text{\emph{T}} X \}< \dim f(X) \leq n \right\},
\end{eqnarray*}
where these sets are respectively:
\begin{itemize}
\item nowhere dense;
\item residual;
\item meagre but dense, unless $n \leq  \dim_\text{\emph{T}} X$ in which case it is empty.
\end{itemize}
\end{thm}

We will prove Theorem \ref{hausdecomp} in Subsection \ref{typproofs1}.

\begin{thm} \label{packdecomp}
Let $\text{\emph{Dim}}$ denote $\dim_\text{\emph{P}}$ or $\overline{\dim}_\text{\emph{B}}$.  Then $C_n(X)$ is a disjoint union of the following three sets:
\begin{eqnarray*}
C_n(X) &=& \left\{ f \in C_n(X)  : 0 \leq \text{\emph{Dim}} f(X) < \min\{n,  \, \dim_\text{\emph{T}} X\} \right\}\\ \\
&\,&  \quad \cup \ \left\{ f \in C_n(X)  :  \min\{n,  \, \dim_\text{\emph{T}} X \} \leq \text{\emph{Dim}} f(X) < n \right\}\\ \\
&\,& \quad \cup \ \left\{ f \in C_n(X)  :  \text{\emph{Dim}} f(X) = n \right\},
\end{eqnarray*}
where these sets are respectively:
\begin{itemize}
\item nowhere dense;
\item meagre but dense, unless $n \leq  \dim_\text{\emph{T}} X$ in which case it is empty;
\item residual.
\end{itemize}
\end{thm}

We will prove Theorem \ref{packdecomp} in Subsection \ref{typproofs2}.  Finally, we obtain precise results on the typical Hausdorff and packing measures of $f(X)$.
\begin{thm} \label{measures}
We have the following dichotomy:
\begin{itemize}
\item[(1)] If $n \leq  \dim_\text{\emph{T}} X$, then for a typical $f \in C_n(X)$, we have
\[
\dim_\text{\emph{P}} f(X) = \dim_\text{\emph{H}} f(X)   = n
\]
and
\[
0\ < \ \mathcal{H}^n( f(X))  =  \mathcal{P}^n( f(X)) \ <  \ \infty.
\]
\item[(2)] If $n >  \dim_\text{\emph{T}} X$, then for a typical $f \in C_n(X)$, we have
\[
\dim_\text{\emph{H}} f(X)  =  \dim_\text{\emph{T}} X,
\]
\,
\[
\dim_\text{\emph{P}} f(X) = n,
\]
\,
\[
\mathcal{H}^{\dim_\text{\emph{T}} X}( f(X))  = \infty
\]
and
\[
\mathcal{P}^n( f(X))   =  0
\]
and, moreover, the measure $\mathcal{H}^{\dim_\text{\emph{T}} X} \vert_{f(X)}$ is not $\sigma$-finite.
\end{itemize}
\end{thm}

We will prove Theorem \ref{measures} in Subsection \ref{typproofs3}.  It is interesting to note that a similar dichotomy was observed in \cite[Theorem 2.5]{me_typrandom} when studying Hausdorff and packing measures of typical random self-similar fractals.

\section{Proofs}

In this section we will prove our main results. All balls and neighbourhoods are assumed to be open unless stated otherwise. We write $B(x,r)$ to denote the open ball centered at $x$ with radius $r$.

\subsection{Proofs concerning the topological structure of $C_n(X)$} \label{keytypproofs}

In this subsection we will prove a sequence of lemmas which will provide a detailed description of the topological structure of $C_n(X)$ in terms of the dimensions of images of $X$.  In the following two subsections we will prove Theorems \ref{hausdecomp} and \ref{packdecomp}, which will follow easily from the results in this subsection.
\\ \\
Recall that, given disjoint sets, $A,B \subseteq X$, a set $P \subseteq X$ is called a \emph{partition} between $A$ and $B$ if there exists open sets $U \supseteq A$ and $V \supseteq B$ such that $U \cap V = \emptyset$ and $P = X \setminus (U \cup V)$.  We will utilise the following result relating partitions to topological dimension, see \cite[Theorem 1.7.9]{engel}.
\begin{prop} \label{partitions}
For a separable metric space $X$, the following statements are equivalent:
\begin{itemize}
\item[(1)] For all collections $(A_1, B_1), \dots, (A_k, B_k)$ of pairs of disjoint closed subsets of $X$  there exists partitions $P_i$ between $A_i$ and $B_i$ such that $\bigcap_{i=1}^{k} P_i = \emptyset$;
\item[(2)] $\dim_\text{\emph{T}} X \leq k-1$.
\end{itemize}
\end{prop}
\begin{lma} \label{nwd1}
The set
\[
\mathcal{N}_1 = \left\{ f \in C_n(X)  : \dim_\text{\emph{T}} f(X) <  \min\{n,  \, \dim_\text{\emph{T}} X \} \right\}
\]
is nowhere dense.
\end{lma}

\begin{proof}
Let $t = \min\{n,  \, \dim_\text{T} X \}$.  Assume to the contrary that for some $f \in C_n(X)$ and $r>0$, $\mathcal{N}_1$ is dense in $B(f,r)$.  Since $f$ is uniformly continuous there exists $\delta>0$ such that if $X_0 \subseteq X$ with $\text{diam}(X_0) \leq \delta$, then $\text{diam}(f(X_0)) < r/t$.  Now decompose $X$ into finitely many compact sets with diameter less than or equal to $\delta$.  Since topological dimension is stable under taking finite (or countable) unions of closed sets, see \cite[Theorem 1.5.3.]{engel}, at least one of the sets in this decomposition has the same topological dimension as $X$.  Fix such a set $X_0 \subseteq X$ with $\text{diam}(X_0) \leq \delta$ and $\dim_\text{T} X_0 = \dim_\text{T} X$ and note that $\text{diam}(f(X_0)) < r/t$.
\\ \\
Let $(A_1, B_1), \dots, (A_t, B_t)$ be arbitrary pairs of disjoint closed subsets of $X_0$.  We will construct partitions $\{P_i\}_{i=1}^{t}$ with $\bigcap_{i=1}^{t} P_i = \emptyset$ from which it follows that $\dim_\text{T} X_0 \leq t-1< \dim_\text{T} X$ which is a contradiction.  Let $f_1, \dots, f_n \in C_1(X)$ be such that $f(x) = (f_1(x), \dots, f_n(x))$ and observe that we may construct a set of functions $\{g_i\}_{i=1}^{t}$ such that
\begin{itemize}
\item[(1)] $g_i \in B(f_i, r/t)$ for each $i \in \{1, \dots, t\}$;
\item[(2)] $g_i(A_i) \cap g_i(B_i) = \emptyset$ for each $i \in \{1, \dots, t\}$;
\item[(3)] There exists functions $g_{t+1}, \dots, g_n \in C_1(X)$ such that the function $g \in C_n(X)$ defined by $g(x) = (g_1(x), \dots, g_t(x), g_{t+1}(x), \dots, g_n(x))$ is such that $g \in \mathcal{N}_1 \cap B(f,r)$.
\end{itemize}
We can do this in the following way. Let $i \in \{1, \dots, t\}$. Since $\text{diam}(f_i(X_0)) \leq \text{diam}(f(X_0)) < r/t$, we can define $g_i$ first on $A_i \cup B_i$, mapping into a sufficiently small neighbourhood of $f_i(X_0)$ such that $g_i(A_i) \cap g_i(B_i) = \emptyset$, and then extend it to the whole of $X$ by Tietze's Extension Theorem so that it satisfies properties (1) and (2). Observe that clearly we can choose $g_{t+1}, \dots, g_n$ such that $g \in B(f, r)$ and we can also assume $g \in \mathcal{N}_1$ since $\mathcal{N}_1$ is dense in $B(f, r)$.
\\ \\
Since $g \in \mathcal{N}_1$, we have $\dim_\text{T} g(X_0) \leq t-1$, and by Proposition \ref{partitions} there exists partitions $Q_i$ between $g(A_i)$ and $g(B_i)$ such that $\bigcap_{i=1}^{t} Q_i = \emptyset$.  Finally, observe that  $P_i := (g \vert_{X_0})^{-1}(Q_i)$ is a partition between $A_i$ and $B_i$ and
\[
\bigcap_{i=1}^{t} P_i \ = \  \bigcap_{i=1}^{t} \, (g \vert_{X_0})^{-1}(Q_i)  \ = \ (g \vert_{X_0})^{-1} \left( \bigcap_{i=1}^{t} Q_i  \right) \ = \  \emptyset
\]
which yields our contradiction.
\end{proof}

\begin{lma} \label{dense1}
The set
\[
\mathcal{D}_1 = \left\{ f \in C_n(X) : \dim_\text{\emph{T}} f(X) = n \right\}
\]
is dense.
\end{lma}

\begin{proof}
It suffices to prove that the set $\left\{ f \in C_n(X) : f(X) \text{ has non-empty interior} \right\}$ is dense.  Fix $f \in C_n(X)$ and $\varepsilon>0$.  Also let $K \subseteq X$ be a set homeomorphic to the Cantor space with the property that $\text{diam}\big(f(K_\delta)\big)< \varepsilon/2$, where $K_\delta$ denotes the $\delta$-neighbourhood of $K$.  Note that we can find such a $K$ by \cite[Theorem 11.11]{realanalysis}, mentioned above, and the continuity of $f$.  Now fix $x \in K$ and observe that $f(K_\delta) \subseteq B(f(x) , \varepsilon/2)$. By \cite[Theorem 4.18]{kechris} we may find a continuous surjection $g_0:K \to B(f(x) , \varepsilon/2)$ and by applying Tietze's Extension Theorem to the coordinate functions we can extend $g_0$ to a map $g \in C_n(X)$ such that
\[
g\vert_{K} = g_0 \qquad \text{and} \qquad g\vert_{X\setminus K_\delta} = f.
\]
It is clear that $g(X)$ has non-empty interior and that $\|f-g \|_\infty < \varepsilon$, which completes the proof.
\end{proof}

\begin{lma} \label{dense2}
The set
\[
\mathcal{D}_2 = \left\{ f \in C_n(X)  : \overline{\dim}_\text{\emph{B}} f(X) \leq \min\{n,  \, \dim_\text{\emph{T}} X \} \right\}
\]
is dense.
\end{lma}

\begin{proof}
We will assume that $\dim_\text{T} X<n$ as otherwise $\mathcal{D}_2 = C_n(X)$ and we are done.  We will first make use of a classical result in dimension theory which states that for a compact metric space $X$ we have
\begin{equation} \label{szp}
\dim_\text{T} X \ = \  \inf \left\{ \overline{\dim}_\text{B} X_0 : X_0 \text{ is homeomorphic to $X$} \right\}
\end{equation}
and this infimum can always be obtained.  This result was originally proved by Szpilrajn in 1937 \cite{szpil} with upper box dimension replaced by Hausdorff dimension.  Szpilrajn's result has been studied and strengthened by numerous authors over the years with the most general version being obtained by Luukkainen \cite{luk} (see also Charalambous \cite{char} and for multifractal analogues, see Olsen \cite{multiszpil}).  By (\ref{szp}) we have that there is a metric space $X_0$ that is homeomorphic to $X$ via $h\colon X\to X_0$ such that $\overline{\dim}_\text{B} X_0 = \dim_\text{T} X$. Then $H\colon C_n(X_0)\to C_n(X)$ defined by $H(f)=f\circ h$ is a homeomorphism between $C_n(X_0)$ and $C_n(X)$. If $g\in C_n(X_0)$ is a Lipschitz function, then 
\[
\overline{\dim}_\text{B} H(g)(X)=\overline{\dim}_\text{B} g(h(X))=\overline{\dim}_\text{B} g(X_0) \leq \overline{\dim}_\text{B} X_0=\dim_\text{T} X,
\]
because Lipshitz functions do not increase the upper box dimension, see \cite[Exercise 3.1]{falconer}, and so $H(g)\in \mathcal{D}_2$. Thus to prove that $\mathcal{D}_2$ is dense in $C_n(X)$, 
it suffices to show that the Lipschitz functions are dense in $C_n(X_0)$. Let $f \in C_n(X_0)$ be such that $f(x) = (f_1(x) , \dots, f_n(x))$ with each $f_i \in C_1(X)$ and fix $\varepsilon>0$.  By the Stone-Weierstrass Approximation Theorem (see, for example, \cite[Theorem 7.30]{rudin1}) we may choose Lipschitz maps, $g_i$, in $C_1(X_0)$ such that
\[
\sup_{x \in X_0}  \ \max_{i =1, \dots, n} \ \lvert g_i(x) - f_i(x) \rvert \, < \, \frac{\varepsilon}{n}
\]
and it is easy to see that the function $g \in C_n(X)$ defined by $g(x) = (g_1(x) , \dots, g_n(x))$ is both Lipschitz and $\varepsilon$ close to $f$ in $C_n(X_0)$, which completes the proof.
\end{proof}

\begin{comment}
Whence, writing $\Gamma: \times_{i=1}^{n} C_1(X_0) \to C_n(X_0)$ for the natural homeomorphism between $\times_{i=1}^{n} C_1(X_0)$ equipped with the supremum norm and $C_n(X_0)$,
\[
C_n(X_0) \   \supseteq  \  \overline{L_n(X_0)} \  \supseteq \   \overline{\Gamma\left(\times_{i=1}^{n} L_1(X_0)\right)} \  = \   \Gamma\left(\times_{i=1}^{n} \overline{L_1(X_0)}\right) \   = \   \Gamma\left(\times_{i=1}^{n} C_1(X_0)\right) \  = \  C_n(X_0)
\]
and so the Lipschitz maps are also dense in $C_n(X_0)$, which completes the proof.
\end{comment}

Before stating and proving the next lemma we will fix some notation.  For a bounded set $F$ and $\delta>0$, let $N_\delta (F)$ denote the smallest number of open sets required for a $\delta$-cover of $F$ and let $M_\delta(F)$ denote the maximum number of sets possible in a $\delta$-packing of $F$, where a $\delta$-packing is a collection of closed disjoint balls with radius $\delta$ and centres in $F$.  Recall that the lower and upper box dimensions of a set $F \subseteq X$ are defined by
\[
\underline{\dim}_\text{B} F = \liminf_{\delta \to 0} \, \frac{\log N_\delta (F)}{-\log \delta}
\qquad
\text{and}
\qquad
\overline{\dim}_\text{B} F = \limsup_{\delta \to 0} \,  \frac{\log N_\delta (F)}{-\log \delta},
\]
respectively, and an equivalent definition is obtained if we replace $N_\delta(F)$ by $M_\delta(F)$.  Let $(\mathcal{K}(X),d_\mathcal{H})$ be the set of non-empty compact subsets of $X$ endowed with the \emph{Hausdorff metric}, that is, $d_\mathcal{H}(K_1,K_2)=\inf\{\delta: K_1\subseteq (K_2)_\delta,~ K_2\subseteq (K_1)_\delta\}$, where $(K)_\delta$ denotes the $\delta$-neighbourhood of a set $K$. The following semicontinuity properties are fundamental and have been noted before, see for example \cite[Lemma 3.1 and Lemma 4.1]{mattilamauldin}, but we include the simple proofs for completeness.

\begin{lma} \label{semicontinuity}
Let $\delta >0$.  The map $N_\delta \colon \mathcal{K}(X) \to \mathbb{R}$ is upper semicontinuous and the map $M_\delta \colon \mathcal{K}(X) \to \mathbb{R}$ is lower semicontinuous.
\end{lma}

\begin{proof}
Fix $\delta >0$.  Since the sets in $\mathcal{K}(X)$ are closed and the covering sets are open, it is clear that
\[
N_\delta^{-1}\left((-\infty, t)\right) = \{ K \in \mathcal{K}(X) : N_\delta(K) < t\}
\]
is open for all $t \in \mathbb{R}$ and so $N_\delta$ is upper semicontinuous.  Similarly, since the sets in $\mathcal{K}(X)$ are closed and the sets used in the packings are closed, it is clear that
\[
M_\delta^{-1}\left((t,\infty)\right) = \{ K \in \mathcal{K}(X) : M_\delta(K) > t\}
\]
is open for all $t \in \mathbb{R}$ and so $M_\delta$ is lower semicontinuous.
\end{proof}

\begin{lma} \label{res1}
The set
\[
\mathcal{R}_1= \left\{ f \in C_n(X)  :   \underline{\dim}_\text{\emph{B}} f(X) \leq \min\{n,  \, \dim_\text{\emph{T}} X \} \right\}
\]
is residual.
\end{lma}

\begin{proof}
If $n \leq  \dim_\text{T} X$, then $\mathcal{R}_1 = C_n(X)$ and we are done, so we may assume that $\dim_\text{T} X < n$ and write $t=\min\{n,  \, \dim_\text{T} X \} = \dim_\text{T} X$.  We will prove that $\mathcal{R}_1$ is a dense $\mathcal{G}_\delta$ subset of $C_n(X)$.
\\ \\
(i) $\mathcal{R}_1$ is $\mathcal{G}_\delta$.  Let $\Lambda \colon C_n(X) \to \mathcal{K}(\mathbb{R}^n)$ be defined by $\Lambda(f) = f(X)$ and observe that it is continuous.  We have
\begin{eqnarray*}
\mathcal{R}_1 &=&  \bigcap_{m=1}^{\infty} \ \bigcup_{\delta \in (0, 1/m)} \left\{f \in C_n(X)  : \frac{\log N_\delta(f(X))}{-\log \delta} < t+\tfrac{1}{m} \right\} \\ \\
&=& \bigcap_{m=1}^{\infty}  \bigcup_{\delta \in (0, 1/m)} \left\{  f \in C_n(X)  : N_\delta(f(X)) < \delta^{-t-1/m} \right\} \\ \\
&=& \bigcap_{m=1}^{\infty}  \bigcup_{\delta \in (0, 1/m)} \Lambda^{-1} \,  N_\delta^{-1} \, \left( \left(-\infty, \,  \delta^{-t-1/m}\right) \right).
\end{eqnarray*}
The set $\Lambda^{-1} \,  N_\delta^{-1} \, \left( \left(-\infty, \,  \delta^{-t-1/m}\right) \right)$ is open by the continuity of $\Lambda$ and the upper semicontinuity of $N_\delta$, see Lemma \ref{semicontinuity}.  It follows that $\mathcal{R}_1$ is a $\mathcal{G}_\delta$ subset of $C_n(X)$.
\\ \\
(ii) $\mathcal{R}_1$ is dense.  This follows immediately since $\mathcal{R}_1 \supseteq \mathcal{D}_2$ and $\mathcal{D}_2$ is dense by Lemma \ref{dense2}.
\end{proof}

Our next goal is to prove that the typical packing dimension is as large as possible, $n$.  However, due to the extra step in the definition of packing measure, packing dimension is often more difficult to work with than Hausdorff dimension.  As such we will first prove an auxiliary result concerning upper box dimension and then deduce the required result for packing dimension.

\begin{lma} \label{res2}
The set
\[
\mathcal{R}_2= \left\{ f \in C_n(X)  :   \overline{\dim}_\text{\emph{B}} f(X)  = n \right\}
\]
is residual.
\end{lma}

\begin{proof}
We will show that $\mathcal{R}_2$ is a dense $\mathcal{G}_\delta$ subset of $C_n(X)$.
\\ \\
(i) $\mathcal{R}_2$ is $\mathcal{G}_\delta$.  Let $\Lambda$ be defined as above.  We have
\begin{eqnarray*}
\mathcal{R}_2 &=& \bigcap_{m=1}^{\infty}  \ \bigcup_{\delta \in (0, 1/m)} \left\{f \in C_n(X)  : \frac{\log M_\delta(f(X))}{-\log \delta} > n-\tfrac{1}{m} \right\} \\ \\
&=& \bigcap_{m=1}^{\infty}  \bigcup_{\delta \in (0, 1/m)} \left\{  f \in C_n(X)  : M_\delta(f(X)) > \delta^{-n+1/m} \right\} \\ \\
&=& \bigcap_{m=1}^{\infty}  \bigcup_{\delta \in (0, 1/m)} \Lambda^{-1} \,  M_\delta^{-1} \, \left( \left( \delta^{-n+1/m}, \, \infty \right) \right).
\end{eqnarray*}
The set $\Lambda^{-1} \,M_\delta^{-1} \, \left( \left( \delta^{-n+1/m}, \, \infty \right) \right)$ is open by the continuity of $\Lambda$ and the lower semicontinuity of $M_\delta$, see Lemma \ref{semicontinuity}.  It follows that $\mathcal{R}_2$ is a $\mathcal{G}_\delta$ subset of $C_n(X)$.
\\ \\
(ii) $\mathcal{R}_2$ is dense.  This follows immediately since $\mathcal{R}_2 \supseteq \mathcal{D}_1$ and $\mathcal{D}_1$ is dense by Lemma \ref{dense1}.
\end{proof}

Before showing that the required result for packing dimension follows from the above lemma, we state a well-known technical lemma.  We give its simple proof for completeness.

\begin{lma} \label{extendingtypical}
Let $X,Y$ be complete metric spaces and let $P \colon X\to Y$ be a continuous open map. If $A \subseteq Y$ is residual then $P^{-1}(A)\subseteq X$ is also residual.
\end{lma}

\begin{proof}
We may assume that $A$ is a dense $\mathcal{G}_\delta$ set in $Y$. The continuity of $P$ implies that $P^{-1}(A)$ is also $\mathcal{G}_\delta$, thus it is enough to prove that $P^{-1}(A)$
is dense in $X$. Let $U \subseteq X$ be a non-empty open set.  It follows that $P(U) \subseteq Y$ is also non-empty and open and hence $P(U) \cap A \neq \emptyset$ and so $U \cap P^{-1}(A) \neq \emptyset$.
Thus $P^{-1}(A)$ is dense in $X$.
\end{proof}

%see \cite[Lemma 5.1]{bairefunctions}.
%\begin{lma}[Hyde et al] \label{extendingtypical}
%Let $X$ and $Y$ be metric space and let $\Phi: X \to Y$ be a map with the proerty that
%\[
%\Phi\left(B(x,r)\right) = B\left(\Phi(x),r\right)
%\]
%for all $x \in X$ and $r>0$.  Then, if $\mathcal{R}$ is a residual subset of $Y$, then $\Phi^{-1}(\mathcal{R})$ is a residual subset of $X$.
%\end{lma}
%\begin{proof}
%This is proved in \cite[Lemma 5.1]{bairefunctions}.
%\end{proof}

\begin{lma} \label{res3}
The set
\[
\mathcal{R}_3 = \left\{ f \in C_n(X)  :   \dim_\text{\emph{P}} f(X)  = n \right\}
\]
is residual.
\end{lma}

\begin{proof}
Since $X$ is uncountable and compact it follows that $X$ contains a closed subset, $K$, homeomorphic to the Cantor space, see \cite[Theorem 11.11]{realanalysis}.  Let $\mathcal{A}(K)$ be a countable family of subsets of $K$ each of which is homeomorphic to $K$ and such that every ball centered in $K$ with positive radius contains a member of $\mathcal{A}(K)$. It is a well-known result in dimension theory that if $F\subseteq \mathbb{R}^n$ is compact and such that $\overline{\dim}_\text{B} (F \cap V) = \overline{\dim}_\text{B} F$ for all open sets $V$ that intersect $F$, then $\dim_\text{P} F =  \overline{\dim}_\text{B} F$, see \cite[Corollary 3.9]{falconer}.  This result together with the fact that each $f \in C_n(X)$ is continuous gives
\begin{eqnarray*}
\mathcal{R}_3 &=&  \left\{ f \in C_n(X)  :   \dim_\text{P} f(X)   \geq n \right\}\\ \\
&\supseteq& \bigcap_{K_0 \in \mathcal{A}(K)} \left\{ f \in C_n(X)  :   \overline{\dim}_\text{B} f(K_0)   \geq n \right\}.
\end{eqnarray*}
Lemma \ref{res2} implies that, for all $K_0 \in \mathcal{A}(K)$, the set $\left\{ f \in C_n(K_0)  :   \overline{\dim}_\text{B} f(K_0) \geq n \right\}$ is a residual subset of $C_n(K_0)$.  Now for each $K_0 \in \mathcal{A}(K)$ define a map $P_{K_0} \colon C_n(X) \to C_n(K_0)$ by $P_{K_0}(f)= f \vert_{K_0}$.
Clearly, $P_{K_0}$ is continuous and if $g_0 \in C_n(K_0)$ is $\varepsilon$ close to $P_{K_0}(f)$ then it has an extension $g \in C_n(X)$ by Tietze's Extension Theorem such that $g$ is $\varepsilon$ close to $f$. Thus $P_{K_0}$ is open.  Hence, it follows from Lemma \ref{extendingtypical} that for all $K_0 \in \mathcal{A}(K)$, the set
\[
\left\{ f \in C_n(X)  :   \overline{\dim}_\text{B} f(K_0) \geq n \right\} \ = \ P_{K_0}^{-1} \left( \left\{ f \in C_n(K_0)  :   \overline{\dim}_\text{B} f(K_0) \geq n \right\} \right)
\]
is a residual subset of $C_n(X)$ which proves that $\mathcal{R}_3$ is residual.
\end{proof}

%In Lemma 3.9 if g_0\in C_n(K_0) is close to P_{K_0}(f) then it has an extension g\in C_n(X) by TET such that g is close to f. Thus P_{K_0} is open. 5

\subsection{Proof of Theorem \ref{hausdecomp}} \label{typproofs1}

The lemmas in Subsection \ref{keytypproofs} combine easily to prove Theorem \ref{hausdecomp}.
\begin{proof}
Let $\dim$ denote $\dim_\text{T}, \, \dim_\text{H}$ or $\underline{\dim}_\text{B}$.  We have
\begin{itemize}
\item[(1)] $\left\{ f \in C_n(X)  : 0 \leq \dim f(X) <  \min\{n,  \, \dim_\text{T} X\} \right\} \ \subseteq  \ \mathcal{N}_1$ and so is nowhere dense by Lemma \ref{nwd1};
\item[(2)] $\left\{ f \in C_n(X)  :  \dim f(X) = \min\{n,  \, \dim_\text{T} X \}\right\} \  \supseteq  \ \mathcal{R}_1 \setminus \mathcal{N}_1$ and so is residual by Lemma \ref{res1} and  \ref{nwd1};
\item[(3)] $\left\{ f \in C_n(X)  :  \min\{n,  \, \dim_\text{T} X \}< \dim f(X) \leq n \right\} \ \subseteq \ C_n(X) \setminus \mathcal{R}_1$ and so is meager by Lemma \ref{res1};
\item[(4)] Assuming $\dim_\text{T} X < n$, we have $\big\{ f \in C_n(X)  :  \min\{n,  \, \dim_\text{T} X \}< \dim f(X) \leq n \big\} \ \supseteq \ \mathcal{D}_1$ which is dense by Lemma \ref{dense1};
\end{itemize}
which completes the proof.
\end{proof}

\subsection{Proof of Theorem \ref{packdecomp}} \label{typproofs2}

The lemmas in Subsection \ref{keytypproofs} combine easily to prove Theorem \ref{packdecomp}.
\begin{proof}
Let $\text{Dim}$ denote $\dim_\text{P}$ or $\overline{\dim}_\text{B}$.  We have
\begin{itemize}
\item[(1)] $\left\{ f \in C_n(X)  : 0 \leq \text{Dim} f(X) < \min\{n,  \, \dim_\text{T} X\} \right\} \ \subseteq  \ \mathcal{N}_1$ and so is nowhere dense by Lemma \ref{nwd1};
\item[(2)] $\left\{ f \in C_n(X)  :  \min\{n,  \, \dim_\text{T} X \} \leq \text{Dim} f(X) < n \right\} \  \subseteq  \ C_n(X) \setminus \mathcal{R}_3$ and so is meager by Lemma \ref{res3};
\item[(3)] Assuming $\dim_\text{T} X < n$, we have $\left\{f \in C_n(X)  :  \min\{n,  \, \dim_\text{T} X \} \leq \text{Dim} f(X) < n \right\} \ \supseteq \ \mathcal{D}_2 \setminus \mathcal{N}_1$ which is dense by Lemma \ref{dense2}, Lemma \ref{nwd1} and the fact that the set difference of a dense set and a nowhere dense set is dense;
\item[(4)] $\left\{ f \in C_n(X)  :  \text{Dim} f(X) = n \right\} \ \supseteq \ \mathcal{R}_3$ and so is residual by Lemma \ref{res3};
\end{itemize}
which completes the proof.
\end{proof}

\subsection{Proof of Theorem \ref{measures}} \label{typproofs3}

Theorem \ref{measures} (1) follows from the following lemma and the fact that $\mathcal{H}^n=\mathcal{P}^n$ for the Borel subsets of $\mathbb{R}^n$ (see \cite[Theorem 6.12]{mattila} and apply the Lebesgue Density Theorem).

\begin{lma} \label{measurelem1}
Suppose $n\leq \dim_\text{\emph{T}} X$.  For a typical $f \in C_n(X)$, we have
\[
\mathcal{H}^n( f(X)) >0
\]
and for all $f \in C_n(X)$, we have
\[
\mathcal{P}^n( f(X)) < \infty.
\]
\end{lma}

\begin{proof}
The fact that $\mathcal{H}^{n}( f(X)) >0$ follows from the work of Szpilrajn \cite{szpil}, where it is shown that $\mathcal{H}^{\dim_\text{T} X}(X) >0$ for any metric space $X$; also see \cite[Theorem VII 3.]{hurwal}.  Since Theorem \ref{typmain} gives that for a typical $f \in C_n(X)$, we have $\dim_\text{H} f(X)  =  \dim_\text{T} f(X) = n$, the result follows.
\\ \\
The fact that $\mathcal{P}^n( f(X)) < \infty$ for all $f \in C_n(X)$ follows immediately from the fact that $f(X)$ is a bounded subset of $\mathbb{R}^n$ and the fact that $n$-dimensional packing measure is a constant multiple of $n$-dimensional Lebesgue measure on $\mathbb{R}^n$, see \cite[Proposition 1.4.5]{cohn}.
\end{proof}

From now on we will be concerned with the case where $\dim_\text{T} X < n$.  Theorem \ref{measures} (2) follows from the two subsequent lemmas and Theorem \ref{typmain}.

\begin{lma} \label{measurelem2}
If $\dim_\text{\emph{T}} X < n$, then the set
\[
\left\{ f \in C_n(X)  :  \mathcal{H}^{\dim_\text{\emph{T}} X}\vert_{f(X)} \text{\emph{ is not $\sigma$-finite}} \right\}
\]
is residual.
\end{lma}

\begin{proof}
Write $t = \dim_\text{T} X$ and let $P \colon C_n(X) \to C_t(X)$ be defined by
\[
P(f)(x) = (f_1(x), \dots, f_t(x))
\]
for $f \in C_n(X)$ given by $f(x) = (f_1(x), \dots, f_n(x))$.  Then $P$ is clearly continuous and open.  Hence it follows from Lemma \ref{extendingtypical} that if a set $\mathcal{R} \subseteq C_t(X)$ is residual, then $P^{-1}(\mathcal{R})$ is a residual subset of $C_n(X)$.
Kato \cite[Theorem 4.6]{kato} proved that if $X$ is a compact metric space with $\dim_\text{T} X = n$, then for a residual set of functions $f \in C_n(X)$, we have
\begin{itemize}
\item[(1)] $\dim_\text{T} f(X)=n$;
\item[(2)] There exists $\mathcal{F}_\sigma$ sets $E_f, S_f \subseteq f(X)$ such that $E_f \cup  S_f = f(X)$, $\dim_\text{T} E_f \leq n-1$ and for all $y\in S_f$ we have that $f^{-1}(y)$ has cardinality continuum.
\end{itemize}
Let $f$ be in this residual set.  Since $\dim_\text{T} f(X)=n$, $E_f$ and $S_f$ are $\mathcal{F}_\sigma$ sets and $\dim_\text{T} E_f \leq n-1$, it follows that $\dim_\text{T} S_f = n$ since the topological dimension of a countable union of closed sets is the supremum of the individual topological dimensions, see \cite[Theorem 1.5.3.]{engel}, and it follows from this that $\mathcal{H}^n(S_f)>0$ by \cite[Theorem VII 3.]{hurwal}.
\\ \\
By the above argument, we can deduce that there exists a residual set $\mathcal{R}_4 \subseteq C_t(X)$ such that for all $f \in \mathcal{R}_4$, there exists a set $S_f \subseteq f(X)$ such that $\mathcal{H}^t(S_f) >0$ and for all $y \in S_f$ we have that $f^{-1}(y)$ has cardinality continuum.  It follows from a result of Hurewicz that, since $t = \dim_\text{T} X < n$, the set
\[
\mathcal{R}_5:=\left\{f\in C_n(X): \forall y\in \mathbb{R}^n,~\#f^{-1}(y)\leq n \right\}
\]
is residual, see \cite[p. 124.]{kuratowski}, and therefore the set
\[
\mathcal{R}_6 \ = \ P^{-1}(\mathcal{R}_4) \ \cap \  \mathcal{R}_5
\]
is a residual subset of $C_n(X)$.  We will now show that $\mathcal{R}_6 \subseteq \{ f \in C_n(X)  :   \mathcal{H}^{\dim_\text{T} X}\vert_{f(X)} \text{ is not $\sigma$-finite} \}$, proving the lemma.  Assume to the contrary and choose $f \in \mathcal{R}_6$ such that $\mathcal{H}^{\dim_\text{T} X}\vert_{f(X)}$ is $\sigma$-finite. Let $W = \{y \in \mathbb{R}^n : y_1 = \cdots = y_t = 0\}$ and identify $W^\perp$ with $\mathbb{R}^t$.  Since $f \in P^{-1}(\mathcal{R}_4)$ it follows that there exists a set $S_{P(f)} \subseteq P(f)(X) \subseteq \mathbb{R}^t$ of positive $\mathcal{H}^{t}$ measure such that for all $y \in S_{P(f)}$, we have that $f^{-1}(W+y)$  has cardinality continuum, but since $f \in \mathcal{R}_5$ it also follows that for all $y \in \mathbb{R}^n$, we have $\#f^{-1}(y)\leq n$ and these two facts together imply that $f(X) \cap (W+y)$  has cardinality continuum for all $y \in S_{P(f)}$.  However,  the classical intersection theorems of Marstrand and Mattila, see \cite[Theorem 10.10.]{mattila}, imply that for $\mathcal{H}^{t}$ almost every $y \in \mathbb{R}^{t}$ the intersection $f(X) \cap (W+y)$ is at most countable, which yields a contradiction.
\end{proof}

\begin{lma} \label{measurelem3}
If $\dim_\text{\emph{T}} X < n$, then the set
\[
\left\{ f \in C_n(X)  :  \mathcal{P}^{n}( f(X)) >0 \right\}
\]
is meagre.
\end{lma}

\begin{proof}
We have
\[
\left\{ f \in C_n(X)  : \mathcal{P}^{n}( f(X)) >0 \right\} =  \bigcup_{m =1}^{\infty} \left\{ f \in C_n(X)  :  \mathcal{P}^{n}( f(X)) >1/m \right\}
\]
so it suffices to show that for each $m \in \mathbb{N}^+$, the set $\mathcal{N}(m) := \left\{ f \in C_n(X)  :  \mathcal{P}^{n}( f(X)) >1/m \right\}$ is nowhere dense.  Fix $f \in C_n(X)$, $r>0$ and write $t = \dim_\text{T} X$.  Since $\mathcal{D}_2$ is dense, see Lemma \ref{dense2}, we may find $g_1 \in B(f,r/2) \cap \mathcal{D}_2$ such that
\[
\overline{\dim}_\text{B} \, g_1(X) \leq \dim_\text{T} X = t.
\]
It follows from \cite[Proposition 3.2]{falconer} that for all $\varepsilon \in (0, n-t)$, there exists a constant $C_\varepsilon>0$ such that
\[
\mathcal{L}^n \left( (g_1(X))_\rho \right) \leq   C_\varepsilon \, \rho^{n-t-\varepsilon}
\]
for all $\rho \in (0,1)$, where $\mathcal{L}^n$ denotes the $n$-dimensional Lebesgue measure and $(g_1(X))_\rho$ denotes the $\rho$-neighbourhood of $g_1(X)$.  Let $\rho<r/2$ and observe that if $g_2 \in B(g_1, \rho) \subseteq B(f,r)$, then $g_2(X) \subseteq (g_1(X))_\rho$.  Since there exists a constant $C(n)$ such that $\mathcal{P}^n(E) = C(n) \, \mathcal{L}^n(E)$ for Borel sets $E \subseteq \mathbb{R}^n$ (see \cite[Proposition 1.4.5]{cohn}), for $g_2 \in B(g_1, \rho) \subseteq B(f,r)$ we have
\[
\mathcal{P}^{n}\left( g_2(X) \right) \ \leq \  \mathcal{P}^{n} \left( (g_1(X))_\rho \right) \ =\ C(n) \,  \mathcal{L}^n  \left( (g_1(X))_\rho \right) \ \leq \  C(n) \, C_\varepsilon \, \rho^{n-t-\varepsilon} \  < \  1/m
\]
for sufficiently small $\rho$, which completes the proof.
\end{proof}

\vspace{6mm}

\begin{centering}

\textbf{Acknowledgements}

\end{centering}

RB is supported by the Hungarian Scientific Foundation grant no.~72655.  AF, JMF and JTH were all supported individually by EPSRC Doctoral Training Grants.  The authors thank Kenneth Falconer for several helpful comments on the exposition of the article and M\'arton Elekes for drawing their attention to the paper \cite{doug}.  Finally, the authors thank an anonymous referee for several helpful suggestions.

\end{document}